\documentclass[11pt,a4paper]{amsart}

\usepackage{caption}
\usepackage{subcaption}

\usepackage[english]{babel}
\usepackage{fancyhdr}
\usepackage{setspace}
\usepackage{amsmath,amsthm}
\usepackage{amssymb, amsfonts}
\usepackage[dvips]{graphicx}
\usepackage[ansinew]{inputenc}
\usepackage{enumerate}
\usepackage{bbm}
\usepackage{verbatim}

\newtheorem{theorem}{Theorem}[section]
\newtheorem{corollary}[theorem]{Corollary}
\newtheorem{lemma}[theorem]{Lemma}
\newtheorem{proposition}[theorem]{Proposition}

\theoremstyle{definition}
\newtheorem*{acknowledgements}{Acknowledgements}

\newtheorem{assumption}[theorem]{Assumption}

\theoremstyle{remark}

\newtheorem{example}{Example}[]

\newcommand{\R}{\mathbbm{R}}
\newcommand{\N}{\mathbbm{N}}

\newcommand{\argmax}{\operatornamewithlimits{argmax}}

\newcommand{\E}{\mathbbm{E}}

\begin{document}
\title{Optimal stopping and control near boundaries}
\author{Pekka Matom\"aki}
\address{Turku School of Economics, Department of Accounting and Finance, 20014 University of Turku, Finland}
\email{pjsila@utu.fi}
\date{\today}

\begin{abstract}
We will investigate the value and inactive region of optimal stopping and one-sided singular control problems by focusing on two fundamental ratios. We shall see that these ratios unambiguously characterize the solution, although usually only near boundaries. We will also study the well-known connection between these problems and find it to be a local property rather than a global one. The results are illustrated by a number of examples.
\end{abstract}

\maketitle

\section{Introduction}
We consider an optimal stopping problem
\begin{align}\label{eq 1}
V(x)=\sup_\tau\E_x\left\{e^{-r\tau}g(X_\tau)\right\},\quad X_0=x,
\end{align}
as well as a singular control problem 
\begin{align*}
W_Z(x)=\sup_Z\E_x\left\{\int_0^\infty e^{-rs}g(X^Z_s)\circ dZ_s\right\},\quad X_0=x,
\end{align*}
where $r>0$ and $g:\R_+\to\R$ is an upper semicontinuous function and $g(X^Z_s)\circ dZ_s$ is understood as in \cite{JaJoZe08}. Furthermore, 
the underlying process $dX_t=\mu(X_t)dt+\sigma(X_t)dW_t$ is a linear, time-homogeneous It\^o diffusion on $\R_+$; $Z$ is a non-negative control process; and $dX^Z=\mu(X^Z_t)dt+\sigma(X^Z_t)dW_t-dZ_t.$ We will elaborate on these assumptions and notations in the sections to come.

\subsection{Core of the study}

We shall unambiguously determine the solution for these problems near the boundaries in quite a simple way. We will do this by focusing on two fundamental ratios. We use these ratios by embedding the Markovian methods and classical
diffusion theory (see e.g. \cite{Salminen85}) into the Beibel-Lerche method (see \cite{BeiLer97, BeiLer00}); for a
similar approach, see also [9].

In the present study, the singular control case is especially interesting since one often has to solve control problems globally when applying the usual variational arguments (cf. \cite{Karatzas83}). However, we are able to find inactive regions and the value function locally, i.e. near the boundary, without having to consider the problem on the whole state space. Moreover, this is the first time, that the author is aware of, that the Beibel-Lerche -method is used in a singular control situation. 

Another interesting finding concerns the renowned connection between one-sided singular control and optimal stopping problems (see e.g. \cite{KarShr84, KarShr85,BoeKoh98}). We will construct the classical associated optimal stopping problem whose value is the derivative of the value of the given singular control problem. We will, however, show that this property only holds true locally, i.e. near the boundaries; as we operate on a level more general than is typical in the literature, this connection does not automatically hold on the whole space. We shall actually see an example where the value of a stopping problem can be the associated value for two different singular control problem on disjoint sets. Most interestingly, this means that we can interpret the renowned connection between the one-sided singular control and optimal stopping problem to be, in general, a local property rather than a global one. Moreover, we get an interesting necessary connection when this connection can hold on the whole state space.

\subsection{Mathematical introduction}

We are using the Markovian method utilising the classical diffusion theory. This technique is based on work \cite{Salminen85}, where it is applied for optimal stopping problems in a linear diffusion setting. Similar technique has been later applied in different situations e.g. in \cite{Alvarez03,AlRa09,Lempa10,CroMor14}. We mix this Markovian method with the Beibel-Lerche method which is originally developed in \cite{BeiLer97, BeiLer00}, and which is later refined e.g. in different problem settings in \cite{LerUru07}, and using jump processes in \cite{Baurdoux07}. See also \cite{GapLer11} for an analysis considering the relations between Beibel-Lerche and free boundary approaches. The power of the Beibel-Lerche method is in the fact that it can identify the continuation region with quite simple arguments in a general context. 

\subsubsection*{Stopping problem}
Consider the optimal stopping problem \eqref{eq 1}. Our study is strongly based on the ratios $g/\psi$ and $g/\varphi$, where $\psi$ and $\varphi$ are the increasing and decreasing fundamental solutions of ODE $(\mathcal{A}-r)u(x)=0$, where $\mathcal{A}$ is the infinitesimal generator of $X$. The reason these ratios are used becomes clear after noticing that the stopping time $\tau_z=\inf\{t\geq0\mid X_t=z\}$, the first hitting time to a state $z$, provides a value (see e.g. II.10 in \cite{BorSal02})
\begin{align*}
\E_x\left\{e^{-r\tau_z}g(X_{\tau_z})\right\}=
\begin{cases}
\frac{g(z)}{\psi(z)}\psi(x),\quad &x\leq z\\
\frac{g(z)}{\varphi(z)}\varphi(x),\quad &x>z.
\end{cases}
\end{align*}
The formulation above suggests that the ratios $g/\psi$ and $g/\varphi$ could strongly dictate the behaviour of $V$. Indeed, one aspect of the paper is to show that the global maximum points of these ratios allow us to write down explicitly the value function near the boundaries, the upper semicontinuity of $g$ being the only restricting requirement. The areas "near the boundaries" that we are interested in are in fact the intervals $(0,z^*)$ and $(y^*,\infty)$, where  $z^*$ and $y^*$ are, respectively, the greatest and the smallest points that globally maximise $g/\psi$ and $g/\varphi$. In practice, this means that in every optimal stopping problem, as soon as one has found the global maximum points $z^*$ and $y^*$, the problem is completely solved near the boundaries. We shall also prove that the two sets, where the ratio $\frac{g}{\psi}$ is increasing and where $\frac{g}{\varphi}$ is decreasing belong to the continuation region. This is an observation that the author has not come across before.

Another aspect of the study is the following. From literature we know that under some weak assumptions $\tau_S$, a hitting time to the set $S=\{x\in \R_+\mid V(x)=g(x)\}$, is an optimal stopping time for \eqref{eq 1} whenever $\tau_S<\infty$ a.s. (see e.g. Theorem 2.7 in \cite{PesShi06}). Furthermore, we know (Theorem 2.4 in \cite{PesShi06}) that $\tau_S\leq \tau^*$ for any stopping time $\tau^*$ that provides the value $V(x)$. However, there have not been too many examples where $\tau_S\neq \tau^*$. In the present paper we offer such examples.

There are naturally also other studies concerning the fundamental ratios $g/\psi$ and $g/\varphi$, which is no surprise as they play central roles in solving (at least one-sided) optimal stopping problems. In \cite{Alvarez03} sufficient conditions are given under which the ratio $g/\psi$ has exactly one global maximum point, after which $g/\psi$ is decreasing, and it is shown that the value function is then unambiguously given everywhere. (The analogous result holds for the ratio $g/\varphi$.) Another direction is considered in \cite{ChrisIrle11}, where it is shown that the points which maximise the ratio $g/h$ for some $r$-excessive function $h$ are in the stopping region $S$. This directly gives that the points that globally maximise $g/\psi$ and $g/\varphi$ are in the stopping region $S$. 
Furthermore, in an extensive study \cite{LamZer13} on the optimal stopping of linear diffusions these ratios have also been inspected. In the aforementioned study (see also \cite{DayKar03,Dayanik08}), the finiteness of the ratios $g/\psi$ and $g/\varphi$ is linked to the finiteness of the value. Also, it is shown how these ratios agree on the boundaries to the ratios created by the value function. Additionally, in \cite{DayKar03}, and later in \cite{Dayanik08} in a more general setting, the optimal stopping time is characterised as a threshold rule relying on the ratios $g/\psi$ and $g/\varphi$ on the boundaries.

\subsubsection*{Singular control problem and connection to stopping problem} 
Taking advantage of the close relations between optimal stopping and singular control problems we shall utilise similar techniques for singular control problems. We will show that the Beibel-Lerche based method also works in a (one-sided) singular control situation, where the ratios $g/\psi'$ and $g/\varphi'$ dictate the value. Maintaining the harmony between a singular control and optimal stopping scene, we see that the solution to a singular control problem is unambiguously characterised on the intervals $(0,z'^*)$ and $(y'^*,\infty)$, though we now need some additional assumptions on the underlying diffusion. Here $z'^*$ and $y'^*$ are, respectively, the greatest and the smallest point that globally maximises $g/\psi'$ and $g/\varphi'$.

It is a well known and greatly studied fact that a singular stochastic control problem has a close relationship with optimal stopping problems. The close connection between a one-sided singular control problem (i.e. either downward or upward control is allowed, but not both) and an optimal stopping problem was already present in the seminal paper by \cite{BatChe66}. Later studies have shown it to hold in general (see \cite{KarShr84, KarShr85, BenRei04}): For a one-sided singular control problem there exists an associated optimal stopping problem such that the derivative of the value of the one-sided singular control problem is the value of the associated optimal stopping problem. 

This connection partially explains the $C^2$-smooth fit condition for the value of a control problem (cf. e.g. \cite{BayEga08}). Since the value of a stopping problem is often $C^1$ and is a derivative of the value of a control problem, we can interpret the $C^2$-condition as an inherited condition from a smooth fit condition of a stopping problem.

This connection can also be generalised to concern two-sided singular control problems (i.e. both downward and upward controls are present). Interestingly, for a two-sided singular control problem there exists an associated \emph{Dynkin game}, such that the derivative of the two-sided singular control problem constitutes the value of the associated Dynkin game (see \cite{KarWang01} and \cite{Boetius05}). 

Also, a relation between a singular control problem and switching problem has been presented in \cite{GuoTome08b,GuoTome08}. This connection states that the value of an singular control problem is also the value of an associated switching problem. Hence, interestingly, with sequential stopping the, somehow subordinate, derivative connection is converted to an equal connection. As a result, understandably, this connection has been identified as a "missing link" between singular control problems and Dynkin games. In addition, there is a connection between Dynkin games and backward stochastic differential equations with two reflecting barriers (see \cite{HamHas06}).

In our study we will link the two studied problems together by inspecting
 the well-known connection between the one-sided singular control problem and the optimal stopping problem on the intervals $(0,z'^*)$ and $(y'^*,\infty)$. In the literature, one usually has sufficient assumptions to ensure that this connection holds everywhere. However, we find that operating on a general level, this connection does not necessarily hold on the whole space, only near the boundaries.

The contents of this study are as follows. In Section \ref{sec prob} we will present the definitions and the optimal stopping problem in detail. This is followed in Section \ref{sec main} by our main results. These results are then illustrated by several short examples in Section \ref{sec example}. Sections \ref{sec control} and \ref{sec connection} are devoted to the singular stochastic control and its relationship to optimal stopping problems.

\section{Optimal stopping problem and definitions}\label{sec prob}

Let $X_t$ be a regular linear diffusion defined on a filtered probability space $(\Omega,\mathcal{F},\{\mathcal{F}_t\},\mathbf{P})$, evolving on $\R_+:=(0,\infty)$. The assumption that the state space is $\R_+$ is done for reasons of convenience --- It could be any interval on $\R$. We assume that $X_t$ is given as the solution of the It\^{o} equation
\begin{align}\label{eq X}
dX_t=\mu(X_t)dt+\sigma(X_t)dW_t,\quad X_0=x,
\end{align}
where $\mu:\R_+\to\R$ and $\sigma:\R_+\to\R_+$ are measurable mappings. We assume that for all $x\in\R_+$ there exists $\varepsilon>0$ such that $\mu$ and $\sigma$ satisfy the condition $\int_{x-\varepsilon}^{x+\varepsilon}\frac{1+|\mu(y)|}{\sigma^2(y)}dy<\infty$. This ensures that \eqref{eq X} has a unique weak solution (see e.g. Chapter V in \cite{KarShr88}). We assume that the boundaries $0$ and $\infty$ are natural, exit, entrance, or killing (for a characterisation of boundaries, see e.g. Section II.1 in \cite{BorSal02}). We understand that if the process hits an exit or a killing boundary, it is sent immediately to a cemetery state $\partial\notin \R_+$, where it stays for the rest of the time (cf. \cite{Dynkin65}, Subsection 3.1). Consequently, boundaries cannot be used as stopping points for a diffusion in any circumstances.
 
We study an optimal stopping problem
\begin{align}\label{eq prob}
V(x)=\sup_\tau\E_x\left\{e^{-r\tau}g(X_\tau)\right\},
\end{align}
where the supremum is taken over all $\mathcal{F}_t$-stopping times and $g:(0,\infty)\to\R$ is an upper semicontinuous reward function that attains positive values  for some $x\in\R_+$ (if $g\leq 0$ always, it is never optimal to stop). Notice that as $0$ and $\infty$ cannot be used as stopping points, $g$ is not necessarily defined on the boundaries and we further understand that $g(\partial)=0$. We denote by $C:=\{x\mid V(x)>g(x)\}$ the continuation region and by $S:=\{x\mid V(x)=g(x)\}$ the stopping region of the considered problem. Furthermore, we follow the notations of \cite{PesShi06} and define \emph{optimal stopping time} to be any stopping time $\tau^*$ at which the supremum \eqref{eq prob} is attained.

Denote by $\psi$ and $\varphi$ the increasing and decreasing fundamental solution to $\left(\mathcal{A}-r\right)u(x)=0$, where $r>0$ and $\mathcal{A}=\frac{1}{2}\sigma^2(x)\frac{d^2}{dx}+\mu(x)\frac{d}{dx}$ is the infinitesimal generator of the diffusion. We will analyse the behaviour of the solution applying the functions $g/\psi$ and $g/\varphi$. To prepare this, we denote by
\begin{itemize}
	\item $M:=\left\{x\mid x=\argmax\{g(x)/\psi(x)\}\right\}$ the set of global maximum points of $g/\psi$ and let $z^*=\sup\{M\}$ be the maximal element of that set; and
	\item $N:=\left\{x\mid x=\argmax\{g(x)/\varphi(x)\}\right\}$ the set of global maximum points of $g/\varphi$ and let $y^*=\inf\{N\}$ be the minimal element of that set.
\end{itemize}
For a boundary point $0$ we understand $g(0)/\psi(0)=\limsup_{x\to 0}g(x)/\psi(x)$ and $g(0)/\varphi(0)=\limsup_{x\to 0}g(x)/\varphi(x)$, and we use analogous interpretation for $\infty$. Especially we notice that with these interpretations $M,N\neq\emptyset$, and that $M$ and $N$ can include $0$ and $\infty$ although $g$ is not necessarily defined on them. Moreover, by upper semicontinuity we have $z^*\in M$ and $y^*\in N$. Further, also by upper semicontinuity, it is true that $\frac{g(b)}{\psi(b)}=\sup_{z}\left\{\frac{g(z)}{\psi(z)}\right\}$ and $\frac{g(a)}{\varphi(a)}=\sup_{y}\left\{\frac{g(y)}{\varphi(y)}\right\}$ for all $b\in M$ and $a\in N$. 

\section{Optimal stopping near boundaries}\label{sec main}
The proofs for the results in this section are given only applying the ratio $g/\psi$, as the proofs with the ratio $g/\varphi$ are treated analogously with obvious changes.

\subsection{Main results}

Let us first show that there are possibly several stopping strategies that provide the values $\frac{g(z^*)}{\psi(z^*)}\psi(x)$ and $\frac{g(y^*)}{\varphi(y^*)}\varphi(x)$.

\begin{lemma}\label{cor 5}
\begin{enumerate}[(A)]
	\item Let $a,b\in M$, $b\neq 0$, and $0\leq a\leq b\leq z^*<\infty$ and let $x\leq b$. Then the following stopping times yield the same value $\frac{g(z^*)}{\psi(z^*)}\psi(x)$:\label{cor 3A}
	\begin{enumerate}[(i)]
	\item $\tau_{z^*}=\inf\{t\geq0\mid X_t=z^*\}$.
	\item $\tau_{b}=\inf\{t\geq0\mid X_t=b\}$.
	\item $\tau_{\{a,b\}}=\inf\{t\geq0\mid X_t\in \{a,b\}\}$, for all $x\in[a,b]$.
	\item $\tau_{M^+}=\inf\{t\geq0\mid X_t\in M^+\}$, where $M^+\subset M$ is any subset that contains $z^*$.
\end{enumerate}
\item Let $a,b\in N$, $a\neq\infty$, and $0<y^*\leq a\leq b\leq \infty$ and let $a\leq x$. Then the following stopping times yield the same value $\frac{g(y^*)}{\varphi(y^*)}\varphi(x)$:\label{cor 3B}
	\begin{enumerate}[(i)]
	\item $\tau_{y^*}=\inf\{t\geq0\mid X_t=y^*\}$.
	\item $\tau_{a}=\inf\{t\geq0\mid X_t=a\}$.
	\item $\tau_{\{a,b\}}=\inf\{t\geq0\mid X_t\in \{a,b\}\}$, for all $x\in[a,b]$.
	\item $\tau_{N^+}=\inf\{t\geq0\mid X_t\in N^+\}$, where $N^+\subset N$ is any subset that contains $y^*$.
\end{enumerate}  
\end{enumerate}
\end{lemma}

\begin{proof}
\noindent\textbf{(A)} \textbf{(i), (ii)} Since $b,z^*\in M$, we have for all $x\leq b$
\[\E_x\left\{e^{-r{\tau_{z^*}}}g(X_{\tau_{z^*}})\right\}=\frac{g(z^*)}{\psi(z^*)}\psi(x)=\frac{g(b)}{\psi(b)}\psi(x)=\E_x\left\{e^{-r{\tau_b}  }g(X_{\tau_{b}})\right\}.\]

\noindent \textbf{(iii)} Let $a\leq x\leq b$. Then the stopping rule $\tau_{\{a,b\}}$ gives a value (cf. Lemma 3.3 in \cite{LamZer13})
\begin{align*}
\E_x\left\{e^{-r\tau_{\{a,b\}}}g(X_{\tau_{\{a,b\}}})\right\}&=\E_x\left\{e^{-r\tau_a};\tau_a<\tau_b\right\}g(a)+\E_x\left\{e^{-r\tau_b};\tau_b<\tau_a\right\}g(b)\\
&=\frac{\varphi(a)g(b)-\varphi(b)g(a)}{\psi(b)\varphi(a)-\varphi(b)\psi(a)}\psi(x)+\frac{\psi(b)g(a)-\psi(a)g(b)}{\psi(b)\varphi(a)-\varphi(b)\psi(a)}\varphi(x).
\end{align*}
Now we can calculate that
\begin{align*}
\psi(b)g(a)-\psi(a)g(b)=\psi(a)\psi(b)\left(\frac{g(a)}{\psi(a)}-\frac{g(b)}{\psi(b)}\right)=0,
\end{align*}
and thus 
\begin{align*}
\E_x\left\{e^{-r\tau_{\{a,b\}}}g(X_{\tau_{\{a,b\}}})\right\}&=\frac{\varphi(a)g(b)-\varphi(b)g(a)}{\psi(b)\varphi(a)-\varphi(b)\psi(a)}\psi(x)\\
&=\frac{\left(\varphi(a)g(b)-\varphi(b)g(a)\right)\psi(b)+\varphi(b)\left(\psi(b)g(a)-\psi(a)g(b)\right)}{\psi(b)\left(\psi(b)\varphi(a)-\varphi(b)\psi(a)\right)}\psi(x)\\
&=\frac{g(b)\left(\psi(b)\varphi(a)-\psi(a)\varphi(b)\right)}{\psi(b)\left(\psi(b)\varphi(a)-\varphi(b)\psi(a)\right)}\psi(x)=\frac{g(b)}{\psi(b)}\psi(x),
\end{align*}
which is the same we got by using a stopping rule $\tau_b$.

\noindent \textbf{(iv)} As a diffusion is continuous, we must have $\tau_{M^+}=\tau_{\{a,b\}}$, where $b$ is the smallest point of $M^+$ such that $b\geq x$, and $a$ is the greatest point of $M^+$ such that $x\leq a$ (if no such $a$ exist, then $\tau_{M^+}=\tau_b$). Consequently the claim follows from part (iii) (or (ii))
\end{proof}

Let us state our main result, which is greatly inspired by Theorem 2 in \cite{BeiLer00}. 

\begin{theorem}\label{prop 1}
\begin{enumerate}[(A)]
	\item For $x\leq z^*$, the value function can be written as 
	\begin{align}\label{eq value}
	V(x)=\frac{g(z^*)}{\psi(z^*)}\psi(x).
	\end{align}
	Moreover, $\tau_{M^+}$ is an optimal stopping time, where $M^+\subset M$ is any subset that contains $z^*$. Lastly, $(0,z^*)\setminus M\subset C$, and $M\subset S$.\label{prop 1 A}
	\item For $x\geq y^*$, the value function can be written as
	\begin{align}\label{eq value phi}
	V(x)=\frac{g(y^*)}{\varphi(y^*)}\varphi(x).
	\end{align}
	Moreover, $\tau_{N^+}$ is an optimal stopping time, where $N^+\subset N$ is any subset that contains $y^*$. Lastly $(y^*,\infty)\setminus N\subset C$, and $N\subset S$. \label{prop 1 Ab}
	\item If $z^*=0$, then there is no $\varepsilon>0$ such that the admissible stopping time $\tau_{(0,\varepsilon)}=\inf\{t\geq0\mid X_t\notin(0,\varepsilon)\}$ yields the value for all $x<\varepsilon$. Similarly if $y^*=\infty$, then there is no $H<\infty$ such that the admissible stopping time $\tau_{(H,\infty)}=\inf\{t\geq0\mid X_t\notin(H,\infty)\}$ yields the value for all $x>H$.\label{prop 1 B}
\end{enumerate}
\end{theorem}

\begin{proof}
\noindent \textbf{(A)}
Let $x\leq z^*$. Then, for all a.s. finite stopping times $\tau$, we have
\begin{align*}
\E_x\left\{e^{-r{\tau}}g(X_\tau)\right\}&=\E_x\left\{e^{-r\tau  }\psi(X_\tau)\frac{g(X_\tau)}{\psi(X_\tau)}\right\}\leq \E_x\left\{e^{-r\tau  }\psi(X_\tau)\right\}\frac{g(z^*)}{\psi(z^*)}=\psi(x)\frac{g(z^*)}{\psi(z^*)}.
\end{align*}

By Lemma \ref{cor 5} this value is attained with a stopping rule $\tau_{M^+}=\inf\{t\geq0\mid X_t\in M^+\}$, where $M^+\subset M$ is any subset containing $z^*$, and thus the proposed $V(x)$ is the value function. 

Lastly, Theorem 2.1 in \cite{ChrisIrle11} says that a point $x\in\R_+$ is in the stopping set if and only if there exists a positive $r$-harmonic function $h$ such that $x\in\argmax\{g/h\}$. This implies straightaway that $M\subset S$.

\noindent \textbf{(C)} Let $M=\{0\}$ and suppose, contrary to our claim, that there exists $\varepsilon>0$ such that $\tau_{(0,\varepsilon)}$ is an optimal stopping time for $x\in(0,\varepsilon)$. Then we have 
\begin{align*}
V(x)&=\E_x\left\{e^{-r{\tau_{(0,\varepsilon)} }}g(X_{\tau_{(0,\varepsilon)}})\right\}=\frac{g(\varepsilon)}{\psi(\varepsilon)}\psi(x)
\end{align*} 
for all $x\in(0,\varepsilon)$. Since $\limsup_{z\downarrow 0}g(z)/\psi(z)$ provides the maximum for $g/\psi$, there exists $\hat{x}\in(0,\varepsilon)$ such that $g(\hat{x})/\psi(\hat{x})>g(\varepsilon)/\psi(\varepsilon)$, whence for all $x\in(0,\hat{x})$ 
\begin{align*}
\E_x\left\{e^{-r{\tau_{(0,\hat{x})}}  }g(X_{\tau_{(0,\hat{x})}})\right\}=\frac{g(\hat{x})}{\psi(\hat{x})}\psi(x)
>\frac{g(\varepsilon)}{\psi(\varepsilon)}\psi(x)=\E_x\left\{e^{-r{\tau_{(0,\varepsilon)}}  }g(X_{\tau_{(0,\varepsilon)}})\right\}=V(x)
\end{align*}
which is contradicts the maximality of $V(x)$.
\end{proof}

Interestingly, if the set $M$ or $N$ contains more than one element, then by Theorem \ref{prop 1} there are several different optimal stopping times that provide the unique value function $V(x)$. Especially, the points of $M$ (and $N$) can now be interpreted as indifference points: For $x\leq z^*$, the decision maker receives the same value irrespectively whether she uses stopping time $\tau_b$, $\tau_{M^+}$, $\tau_{\{a,b\}}$, $\tau_M$, or $\tau_{z^*}$, where $a,b\in M$ are any points for which $a\leq x\leq b$. However, it is quite clear that $\tau_M$ is the smallest (a.s.) of all these stopping times (cf. Theorem 2.4 in \cite{PesShi06}). In the sequel we wish to be unambiguous and thus we select $\tau_M$ to be our optimal stopping time on interval $(0,z^*]$ with a notion that there might also be others.

It should also be mentioned that the values $\limsup_{x\to0}\frac{g(x)}{\varphi(x)}$ and $\limsup_{x\to\infty}\frac{g(x)}{\psi(x)}$ are crucial when investigating the optimality of a stopping time $\tau_S=\inf\{t\geq0\mid X_t\in S\}$ on the whole state space. This question has been treated quite comprehensively in \cite{DayKar03, Dayanik08}. In addition, in Theorem 6.3(III) in \cite{LamZer13} it is proven that $\lim_{n\to\infty}\tau_S\land \tau_{(\alpha_n,\beta_n)}$ provides the value function, where $\alpha_n$ and $\beta_n$ are suitable chosen sequences such that $\alpha_n\to0$, $\beta_n\to\infty$ when $n\to\infty$. In our research this aspect is omitted, as the above mentioned analyses are quite exhaustive, and the present study does not bring any new insights to that subject.

\subsection{Minor results}

In short, Theorem \ref{prop 1} guarantees that if we can find the global maximum points of $g/\psi$ and $g/\varphi$, then the solution is unambiguously characterised near the boundaries. Usually the set $M$ contains only one element, but this is not always the case, as will be illustrated with examples in Section \ref{sec example} below. The theorem also gives rise to a handful of corollaries, which shall be presented in this subsection (proofs are given in the Subsection \ref{subsec corollary proofs}). These corollaries are more or less straight consequences of Theorem \ref{prop 1} (and hence of Theorem 2 from \cite{BeiLer00}), but they have not been written out explicitly before.

Let us start with the ordering of $z^*$ and $y^*$.

\begin{corollary}\label{cor 2b}
One has $z^*\leq y^*$.
\end{corollary}

This means that the regions where the value is dictated by ratios $g/\psi$ and $g/\varphi$ are always separated.

Let us then present an easy, but surprisingly powerful corollary. 
\begin{corollary}\label{cor 2}
\begin{enumerate}[(A)]
	\item If $z^*>0$. Then, for all $x\leq z^*$, the value reads as in \eqref{eq value}, $\tau_M$ is the optimal stopping time and $(0,z^*)\setminus M\subset C$.\label{cor 2 A}
	\item If $y^*<\infty$. Then, for all $x\geq y^*$, the value reads as in \eqref{eq value phi}, $\tau_N$ is the optimal stopping time, and $(y^*,\infty)\setminus N\subset C$.\label{cor 2 B}
\end{enumerate}
\end{corollary}

The power of this corollary lies in the fact that if one can check somehow that $z^*>0$, the value is immediately characterised near the boundary $0$; For example if $g(0+)<0$ or $\frac{d}{dx}\frac{g}{\psi}(0+)>0$, then we know at once that $z^*>0$. In this way Corollary \ref{cor 2} can be used to quicken the standard Beibel-Lerche method: If $z^*>0$, we immediately know the solution without needing more closely inspection. This corollary is also related to corollaries 7.2 and 7.3 in \cite{LamZer13}, where it is shown how the value function looks like if we find out that $(0,z)\subset C$ and $z\in S$ (and analogously for $(y,\infty)\subset C$ and $y\in S$). 

In the next two corollaries we will study more closely the situations that $M$ or $N$ include $0$ or $\infty$. 
\begin{corollary}\label{cor b1}
If $z^*=0$ and $y^*=\infty$, then one of the following is true.
\begin{enumerate}[(i)]
	\item It is optimal to stop immediately, i.e. $S=\R_+$ and $g$ is $r$-excessive.
	\item The optimal stopping rule is at least two-boundary rule (i.e. there are $0<a<b<\infty$ such that $(0,a]\cup [b,\infty)\subset S$).
	\item The optimal stopping time is not finite/admissible.
\end{enumerate}
\end{corollary}

\begin{corollary}\label{cor 3}
\begin{enumerate}[(A)] 
	\item Let $z^*=\infty$. Then $V(x)=A\psi(x)$ for all $x\in\R_+$, where $A=\limsup_{x\to\infty}\left\{\frac{g(x)}{\psi(x)}\right\}$. Especially, if $A=\infty$, then $V(x)\equiv\infty$. \label{cor last B}
	\begin{enumerate}[(i)]
	\item If $M=\{z^*\}=\{\infty\}$, then $C=\R_+$ and there does not exist a finite stopping time that yields the value. 
	\item If there is at least one other element in $M$, then $C=\R_+\setminus M$ and, for all $x\leq b^*:=\sup\{M\setminus \{\infty\}\}$, the stopping time $\tau_{M\setminus\{\infty\}}$ is optimal and admissible. For $x\in(b^*,\infty)$ there does not exist an admissible stopping time that yields the value.\label{cor last Bii} \end{enumerate}
	\item Let $y^*=0$. Then $V(x)=B\varphi(x)$, where $B=\limsup_{x\to0}\left\{\frac{g(x)}{\varphi(x)}\right\}$. Especially, if $B=\infty$, then $V(x)\equiv\infty$.\label{cor last D}
	\begin{enumerate}[(i)]
	\item If $N=\{y^*\}=\{0\}$, then $C=\R_+$ and there does not exist a finite stopping time that yields the value.
	\item If $y^*=0$ and there is at least one other element in $N$, then $C=\R_+\setminus N$ and, for all $x\geq a^*:=\inf\{N\setminus \{0\}\}$, the stopping time $\tau_{N\setminus\{0\}}$ is optimal. For $x\in(0,a^*)$ there does not exist an admissible stopping time that yields the value. \label{cor last Dii}
\end{enumerate}
\end{enumerate}
\end{corollary}

We observe that the condition $z^*=\infty$ alone is not enough to guarantee that a stopping region is an empty set nor that a value cannot be attained with an admissible stopping time. Indeed, we will illustrate these cases below in Section \ref{sec example}. It is worth mentioning that the sufficient conditions in Corollary \ref{cor 3} for the value function to be infinite are known results (e.g. Theorem 1 in \cite{BeiLer00}, Theorem 6.3(I) in \cite{LamZer13}, Proposition 5.10 in \cite{DayKar03}).

Let us then show that if $M$ (or $N$) includes an interval, then $g$ must equal to a fundamental solution on it.

\begin{lemma}\label{lemma 2}
\begin{enumerate}[(A)]
	\item If an interval $(a,b)\subset M$, then $g(x)=K \psi(x)$ for all $x\in(a,b)$ for some $K>0$.
	\item If an interval $(a,b)\subset N$, then $g(x)=K \varphi(x)$ for all $x\in(a,b)$ for some $K>0$.
\end{enumerate}
 \end{lemma}

\begin{proof}
\noindent \textbf{(A)} Since $(a,b)\subset M$ we know that $x=\argmax\{g(y)/\psi(y)\}$ for all $x\in(a,b)$. This in turns means that $\frac{g(x)}{\psi(x)}=K$ for some $K>0$ on $(a,b)$. 
\end{proof}

To end the subsection, we will show that using the ratios $g/\psi$ and $g/\varphi$ we can also say something about the continuation region outside the set $(0,z^*)\cup(y^*,\infty)$.

\begin{lemma}\label{lemma increasing}
\begin{enumerate}[(A)]
	\item Denote by $C_\psi:=\left\{x\mid \frac{g(x)}{\psi(x)} \text{ is strictly increasing }\right\}$. Then $C_\psi\subset C$.
	\item Denote by $C_\varphi:=\left\{x\mid \frac{g(x)}{\varphi(x)} \text{ is strictly decreasing }\right\}$. Then $C_\varphi\subset C$.
\end{enumerate}
\end{lemma}

\begin{proof}
\noindent\textbf{(A)} Let $x\in C_\psi$. Since $g/\psi$ is strictly increasing at $x$, we can choose $z>x$ such that $g(z)/\psi(z)>g(x)/\psi(x)$. Now, using a stopping time $\tau_z$, we get
\[V(x)\geq \E_x\left\{e^{-r{\tau_z}  }g(X_{\tau_x})\right\}=\frac{g(z)}{\psi(z)}\psi(x)>\frac{g(x)}{\psi(x)}\psi(x)=g(x)\]
and so $x\in C$.
\end{proof}

Consider a reverse result: An interval $(a,b)$ is subset of $C$, with $a,b\in S$, only if $g/\psi$ is increasing or $g/\varphi$ is decreasing for some $x\in(a,b)$. If this would be true, then one could try to identify all the continuation regions using only ratios $g/\psi$ and $g/\varphi$. However, it is possible to construct an example where $C\neq \emptyset$ and $g/\psi$ and $g/\varphi$ are decreasing and increasing, respectively, everywhere. In other words this reverse result is not true, and thus this approach cannot be used to get an alternative method for identifying continuation and stopping regions.

\subsection{Extensions}

\subsubsection{Involving an integral}
The problem setting can easily be extended to contain also an integral term. To that end, let $\pi:\R_+\to\R$ be a measurable function such that $\E_x\left\{\int_0^\infty e^{-rs}|\pi(X_s)|ds\right\}<\infty$, and let us consider a problem 
\begin{align*}
\sup_\tau\E_x\left\{\int_0^\tau e^{-rs}\pi(X_s)ds+e^{-r\tau}g(X_\tau)\right\}.
\end{align*}
Using a strong Markov property, this can be rewritten as (cf. (1.13) in \cite{LamZer13})
\begin{align*}
(R_r\pi)(x)+\sup_\tau\E_x\left\{e^{-r\tau}\left[g(X_\tau)-(R_r\pi)(X_\tau)\right]\right\},
\end{align*}
where $(R_r\pi)(X_t)=\E_x\left\{\int_0^\infty e^{-rs}\pi(X_s)ds\right\}$ is the resolvent of $\pi$, or a cumulative net present value of $\pi$. We see at once that all the results above hold for this problem with obvious changes and with the sets $M=\left\{x\mid x=\argmax\{\frac{g(z)-(R_r\pi)(z)}{\psi(z)}\}\right\}$ and $N=\left\{x\mid x=\argmax\{\frac{g(y)-(R_r\pi)(y)}{\varphi(y)}\}\right\}$.

\subsubsection{State dependent discounting}

Another fairly straightforward extension can be made by introducing a discounting function: Instead of a constant $r>0$, we can define $r:\R_+\to \R_+$ to be a continuous function, which is bounded away from zero, whence $\int_0^t r(X_s)ds$ is the cumulative discounting from now until a time $t$. As the ODE $\left(\mathcal{A}-r(x)\right)u(x)=0$ has an increasing $\psi_r$ and a decreasing $\varphi_r$ as its two independent solutions, we see that all the previous results hold in this case.

\subsubsection{Other boundaries}

Assume that the boundaries are either absorbing or reflecting and that $g$ is extended to be defined also on these boundaries.

For a reflecting boundary Lemma \ref{cor 5}, Theorem \ref{prop 1}\eqref{prop 1 A},\eqref{prop 1 Ab} and Corollary \ref{cor 3} can quite easily seen to hold as well as Lemmas \ref{lemma 2} and \ref{lemma increasing}. However, as the boundaries can now be used as stopping points, Theorem \ref{prop 1}\eqref{prop 1 B} is no longer valid (for an counterexample, see Example 7 in Section \ref{sec example} below), nor are its corollaries.

The process is trapped in an absorbing boundary once it hits there, and so we have to take into account the positiveness of the reward function at the boundaries. If $g(0),g(\infty)\leq 0$, it is not worthwhile to stop at the boundaries. Consequently, in this case, absorbing boundaries behaves as killing ones and all the preceding results are valid with absorbing boundaries.

However, if $g(0)$ or $g(\infty)$ is positive with absorbing boundaries, then the boundary points are always stopping points. This influences the analysis and consequently, of the introduced results, only Lemmas \ref{lemma 2} and \ref{lemma increasing} hold true in the present formulation. (For a formulation of a value function near absorbing boundaries, see also Corollaries 7.2 and 7.3 in \cite{LamZer13}.)

\subsection{Proofs to the corollaries}\label{subsec corollary proofs}

\begin{proof}[Proof of Corollary \ref{cor 2b}]
Suppose, contrary to our claim, that $y^*<z^*$. Then for all $x\in(y^*,z^*)$, by Theorem \ref{prop 1}, both $\tau_{y^*}$ and $\tau_{z^*}$ would give the value, i.e.
\begin{align*}
\frac{g(y^*)}{\varphi(y^*)}\varphi(x)=\E_x\left\{e^{-r{\tau_{y^*}}  }g(X_{\tau_{y^*}})\right\}=V(x)=\E_x\left\{e^{-r{\tau_{z^*}}  }g(X_{\tau_{z^*}})\right\}=\frac{g(z^*)}{\psi(z^*)}\psi(x).
\end{align*}
From this we get 
\[\frac{\psi(x)}{\varphi(x)}=\frac{g(y^*)}{g(z^*)}\frac{\psi(z^*)}{\varphi(y^*)}.\]
But as $\psi$ is increasing and $\varphi$ is decreasing, the ratio $\psi/\varphi$ cannot be constant for all $x\in(y^*,z^*)$.
\end{proof}

\begin{proof}[Proofs of Corollary \ref{cor 2} and \ref{cor b1}] Straight consequences from Theorem \ref{prop 1}.
\end{proof}

\begin{proof}[Proof of Corollary \ref{cor 3}]
\noindent \textbf{(A)} Since $z^*=\infty$, we have $A=\limsup_{x\to\infty}\frac{g(x)}{\psi(x)}=\sup\left\{\frac{g(x)}{\psi(x)}\right\}$. Clearly $V_A(x)=A\psi(x)$ is an $r$-excessive majorant of $g$, and consequently a candidate for the value.

Let us create an increasing sequence $z_n$ such that $\lim_{n\to\infty}z_n=z^*=\infty$, that the sequence $\frac{g(z_n)}{\psi(z_n)}$ is increasing, and that $\lim_{n\to\infty}\frac{g(z_n)}{\psi(z_n)}=A$. Fix $x\in\R_+$ and let $n_x\in \N$ be such that $z_n>x$ for all $n>n_x$. Then the sequence of values
\begin{align*}
V_n(x)=\E_x\left\{e^{-r{\tau_{z_n}}  }g(X_{\tau_n})\right\}=\frac{g(z_n)}{\psi(z_n)}\psi(x)
\end{align*}
is increasing for all $n>n_x$. Moreover, for each $\varepsilon>0$ there exists $n_\varepsilon\in\N$ such that \[V_A(x)-V_n(x)=\left(A-\frac{g(z_n)}{\psi(z_n)}\right)\psi(x)<\varepsilon,\quad \text{ for all } n>n_\varepsilon.\] 
Therefore $V_A(x)=\lim_{n\to\infty}V_n(x)$ is the limit of the increasing sequence of the values of the admissible stopping times $\tau_{z_n}$, and thus it is the value. 

\noindent \textbf{(A) (i)} The value reads as $A\psi(x)$, but as $\infty$ cannot be used as a stopping point, this value cannot be reached by any finite stopping time. 

\noindent \textbf{(A) (ii)} By Theorem \ref{prop 1}, for all $x\in(0,b^*)$, $\tau_{M\setminus\{\infty\}}$ provides the value. The rest follows from part (i).
\end{proof}

\section{Examples}\label{sec example}
Let us illustrate the results with a geometric Brownian motion on $\R_+$. Now $X_t$ satisfies the stochastic differential equation
\[dX_t=\mu X_tdt+\sigma X_tdW_t,\]
where $\mu\in \R$ and $\sigma>0$, and the boundaries are natural. The fundamental solutions are $\psi(x)=x^{\gamma^+}$ and $\varphi(x)=x^{\gamma^-}$, where $\gamma^+$ is the positive root and $\gamma^-$ the negative root of the characteristic equation
\[\frac{1}{2}\sigma^2\gamma(\gamma-1)+\mu\gamma-r=0.\]
We shall demonstrate the results of Section \ref{sec main} numerically by choosing $\mu=0.1$, $\sigma=0.2$ and $r=0.24$, so that $\psi(x)=x^2$ and $\varphi(x)=x^{-6}$.

\begin{example}
We will present a simple example, where $M$ have more than one point so that there are multiple different optimal stopping rules. By Theorem \ref{prop 1} this is true, if $\frac{g}{\psi}$ have at least two maximum points. To get such a function, choose 
\begin{align*}
g_1(x)=\begin{cases}
x-3,\quad&x\leq 10\\
a_1x-b_1,\quad &x> 10,
\end{cases}
\end{align*}
where $a_1=2\frac{1}{3}$ and $b_1= 16\frac{1}{3}$ are chosen so that $g_1$ is continuous and that $M$ contains two points. With these choices $M=\{6,14\}$, $\sup\left\{{g(x)}/{\psi(x)}\right\}=\frac{1}{12}$ and $V(x)=\frac{1}{12}\psi(x)$ for all $x\leq 14$ (see Figure \ref{fig 1}). For $x<6$, the usually accepted optimal stopping time is $\tau_6=\inf\{t\geq0\mid g(x)=V(x)\}$, but also $\tau_{14}$ is optimal despite the fact that $\tau_{14}>\tau_6$ a.s.

\begin{figure}[!ht]
\begin{center}
\includegraphics[width=0.45\textwidth]{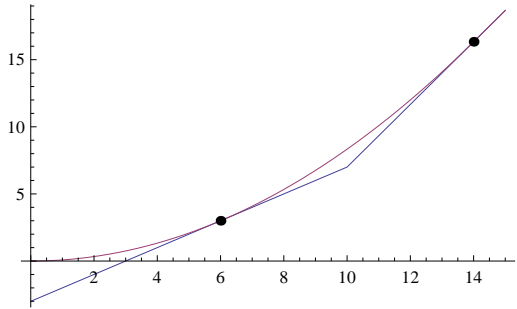}
\caption{\small The functions $V_1$ and $g_1$. The points of $M$ are denoted by black dots. The point $6$ is a singleton stopping point, and it is an indifference point, where the value function coincides with the reward function. }\label{fig 1}
\end{center}
\end{figure}
\end{example}

\begin{example} Extending the previous example, we now show that $M$ can be uncountable, so that there are uncountably many different optimal stopping rules. This is possible (by Theorem \ref{prop 1}) if $M$ contains an interval, which by Lemma \ref{lemma 2} means that $g(x)=K\psi(x)$ for some $K$ on some interval. To that end, choose 
\begin{align*}
g_2(x)=\begin{cases}
x-3,\quad &x\leq 6\\
\frac{1}{12}x^2,\quad &6< x\leq 10\\
1\frac{2}{3}x-8\frac{1}{3},\quad& 10<x\leq 14\\
3x-27,\quad &x> 14.
\end{cases}
\end{align*}
With this choice, we can calculate that $\sup\{g(x)/\psi(x)\}=\frac{1}{12}$, and it is reached at the points $[6,10]\cup \{18\}=M$, so that $V(x)=\frac{1}{12}\psi(x)$ for all $x\leq 18$ (see Figure \ref{fig 2}). Further, by Theorem \ref{prop 1} for $x<18$ any stopping time $\tau_{\{a,18\}}$, $a\in[6,10]$ is optimal.

\begin{figure}[!ht]
\begin{center}
\begin{subfigure}[b]{0.45\textwidth}
\begin{center}
\caption{}
\includegraphics[width=\textwidth]{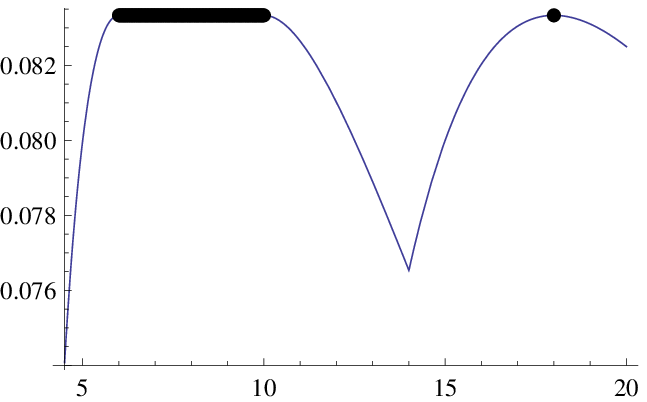}
\end{center}
\end{subfigure}
\begin{subfigure}[b]{0.45\textwidth}
\begin{center}
\caption{}
\includegraphics[width=\textwidth]{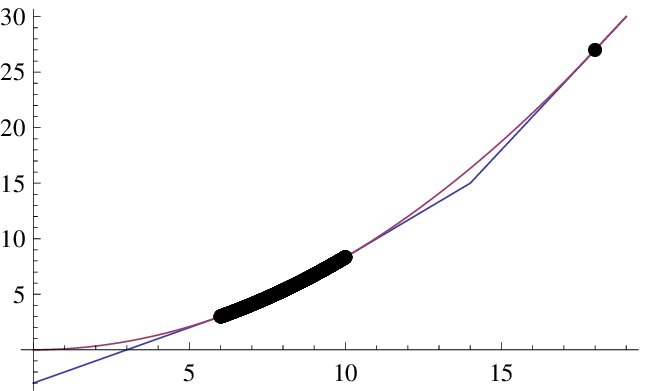}
\end{center}
\end{subfigure}
\caption{\small (A) The function $g_2/\psi$. (B) The functions $V_2$ and $g_2$. The points of $M$ are denoted by black dots. Now $g_2$ coincides with $\frac{1}{12}\psi$ on $(6,10)$, and thus this whole interval maximizes $g_2/\psi$ and can be interpreted as an indifference region for the decision maker.}\label{fig 2}
\end{center}
\end{figure}
\end{example}

\begin{example} In this example we consider a case illustrating that $M$ can include $\infty$, and that the value function is not attainable although the stopping set $S\neq\emptyset$. Choose
\begin{align*}
g_3(x)=
\begin{cases}
x-3,\quad&x<12\\
\frac{1}{12}x^{2-\frac{1}{x^3}}-b_3x^{-5},\quad &x\geq12,
\end{cases}
\end{align*}
where $b_3\approx 742\,205$ is chosen so that $g_3$ is continuous. Now we see that $g_3(x)/\psi(x)\to \frac{1}{12}$ as $x\to\infty$, and so $M=\{6,\infty\}$, and that the value function $V_3(x)=\frac{1}{12}\psi(x)$, for all $x\in\R_+$, exists finitely. However, for $x>6$ there is no admissible stopping time that provides this value (see Figure \ref{fig 3A}).  Notice that for $x\leq 6$, $\tau_6$ is an optimal stopping time. Worth observing is that $g_3$ does not satisfy the integrability condition $\E_x\left\{\sup_s \left\{e^{-rs}g_3(X_s)\right\}\right\}<\infty$, a sufficient condition for the existence of a finite solution. For a similar example, see Example 8.2 in \cite{LamZer13}.

\begin{figure}[ht!]
\begin{center}
\begin{subfigure}[b]{0.45\textwidth}
\begin{center}
\caption{}\label{fig 3A}
\includegraphics[width=\textwidth]{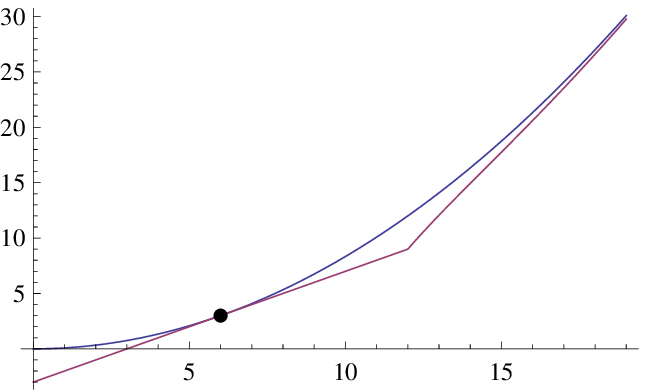}
\end{center}
\end{subfigure}
\begin{subfigure}[b]{0.45\textwidth}
\begin{center}
\caption{}\label{fig 3B}
\includegraphics[width=\textwidth]{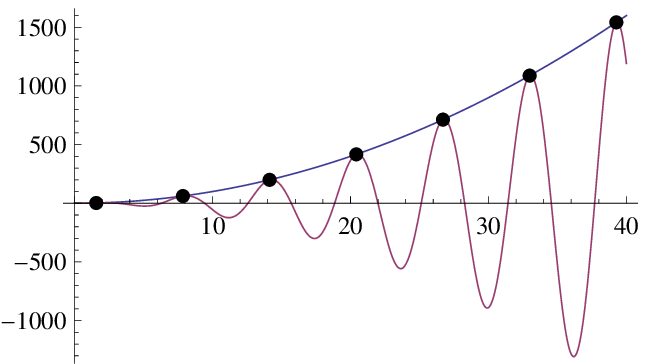}
\end{center}
\end{subfigure}
\caption{\small (A) The functions $V_3$ and $g_3$. (B) The functions $V_4$ and $g_4$. All the stopping points are denoted by black dots. In (A), the only finite stopping point is 6 and in (B) there are infinite amount of stopping points.}\label{fig 3}
\end{center}
\end{figure}
\end{example}

\begin{example}\label{ex sin} In this example we characterise a situation where $z^*=\infty$ and the value is nevertheless attained with a \emph{finite stopping time} for all $x\in \R_+$. To that end, let us choose
\begin{align*}
g_4(x)=x^2\text{sin}(x).
\end{align*}
Now clearly $g_4(x)/\psi(x)=\text{sin}(x)$, and thus $M=\{\frac{\pi}{2}+k2\pi\}_{k\in\N}$, and $z^*=\infty$. The value reads as $V_4(x)=\psi(x)$ for all $x\in \R_+$ and it is attained with a stopping time $\tau_M$, which is finite a.s. (see Figure \ref{fig 3B}). In fact, by Theorem \ref{prop 1}, for each $x$ there are countably many finite stopping times that provide the value $V_4$. Notice that the value is attainable, although, like in the previous example, the integrability condition $\E_x\left\{\sup_s \left\{e^{-rs}g_3(X_s)\right\}\right\}<\infty$ does not hold in this one either.
\end{example}

\begin{example}\label{ex pf} From Corollary \ref{cor 2b} we know that $z^*\leq y^*$. Let us now illustrate that $z^*$ can equal to $y^*$. For that end, let $g_{\ref{ex pf}}(x)=\min\{\psi(x),\varphi(x)\}$. Now $V_{\ref{ex pf}}(x)=g_{\ref{ex pf}}(x)$, $M=(0,1]$, $N=[1,\infty)$, and consequently $z^*=1=y^*$.
\end{example}

\begin{example}\label{ex 6} In this example we show that if a boundary point can be used as a stopping point, previous results do not necessarily hold. Let the state space be $[1,\infty)$ and let $1$ be a reflecting boundary. Let
\begin{align*}
g_{\ref{ex 6}}(x)=\begin{cases}
1,\quad &x=1\\
0,\quad & x\in(1,5)\\
1,\quad &x\geq 5.
\end{cases}
\end{align*}
Then $g_{\ref{ex 6}}$ is an upper semicontinuous function, $M=\{1\}$, and $z^*=1$. It can be easily seen that $\tau_{(1,5)}=\inf\{t\geq0\mid X_t\notin (1,5)\}$ is an optimal stopping time. This differs from Theorem \ref{prop 1}\eqref{prop 1 B}, which says that $\tau_{(1,1+\varepsilon)}$ cannot provide optimal value for any $\varepsilon>0$ for unattainable boundaries. Hence we see that as soon as a boundary can be used as a stopping point, Theorem \ref{prop 1}\eqref{prop 1 B} does not necessarily hold.
\end{example}

\begin{example}\label{ex 7} Here we will study more closely the subtle issue about a continuation region and an optimal stopping time. Assume that $z^*=0$. Then by Theorem \ref{prop 1}\eqref{prop 1 B} we know that $\tau_{(0,\varepsilon)}$ is not an optimal stopping time for any $\varepsilon>0$. From this one could wrongly conclude that $(0,\varepsilon)\subset S$ for some $\varepsilon>0$. However, contrary to this belief, we will illustrate that for some $\varepsilon>0$ we can actually have $(0,\varepsilon)\subset C$, in spite of the fact that $\tau_{(0,\varepsilon)}$ is not an optimal stopping time. 

Let
\begin{align*}
g_{\ref{ex 7}}(x)=
\begin{cases}
x^{-6+x},\quad& x<1\\
x^{-1},\quad & x\geq1,
\end{cases}
\end{align*}
where we have chosen $g_{\ref{ex 7}}$ so that $\lim_{a\to0}\frac{g_{\ref{ex 7}}(a)}{\varphi(a)}>0$. It is easy to check that $M=\{0\}$ and $N=\{\infty\}$ implying that $z^*=0$ and $y^*=\infty$. Then by Theorem \ref{prop 1}\eqref{prop 1 B} we know that $\tau_{(0,\varepsilon)}$ is not an optimal stopping time for any $\varepsilon>0$. However, it can be proven that (cf. Theorem 6.3(III) and Corollary 7.2 in \cite{LamZer13}), for $x\leq 0.1$ say, 
\begin{align*}
V_{\ref{ex 7}}(x)&=\sup_b\lim_{n\to\infty}\E_x\left\{e^{-r\left(\tau_{(\frac{1}{n},b)}\right)}g_{\ref{ex 7}}\left(X_{\tau_{(\frac{1}{n},b)}}\right)\right\}\\
&=\sup_b\left\{\frac{g_{\ref{ex 7}}(0)}{\varphi(0)}\varphi(x)+\left(g_{\ref{ex 7}}(b)-\frac{g_{\ref{ex 7}}(0)}{\varphi(0)}\varphi(b)\right)\psi(x)\right\},
\end{align*}
which is maximized with $b^*\approx 1.22$. Thus $V_{\ref{ex 7}}(0.1)\approx 10^6>794\,328=g_{\ref{ex 7}}(0.1)$. In other words, for $x<b^*$ we have $V_{\ref{ex 7}}(x)>g_{\ref{ex 7}}(x)$ so that $(0,b^*)\subset C$. Moreover, the time $T_\infty:=\lim_{n\to\infty}\tau_{(\frac{1}{n},b^*)}$ provides the value $V(x)$, but now $T_\infty$ is not an admissible stopping time.

We could have also used Lemma \ref{lemma increasing} to prove the point: Now $\frac{g_{\ref{ex 7}}(x)}{\varphi(x)}$ is strictly decreasing on $(0,0.37)$, meaning that $(0,0.37)\subset C$.
\end{example}

\begin{example}\label{ex 7+1}
In all the preceding examples we have had $V'(b)=g'(b)$ for all $b\in M$. Here we will show that actually we do not need smooth fit in order to apply Theorem \ref{prop 1}. To that end, let us take 
\begin{align*}
g_{\ref{ex 7+1}}(x)=\lfloor x \rfloor^2
\end{align*}
to be a step function, which is an upper semicontinuous function. In this case, $M=\{1,2,3,\ldots\}$, $z^*=\infty$ and $V_{\ref{ex 7+1}}(x)=\psi(x)$. Moreover, now $V_{\ref{ex 7+1}}(b)=g_{\ref{ex 7+1}}(b)$ for all $b\in M$, but
unlike in the preceding examples the smooth fit does not hold in this case. See Figure \ref{fig 7+1} for an illustration. 

\begin{figure}[!ht]
\begin{center}
\includegraphics[width=0.45\textwidth]{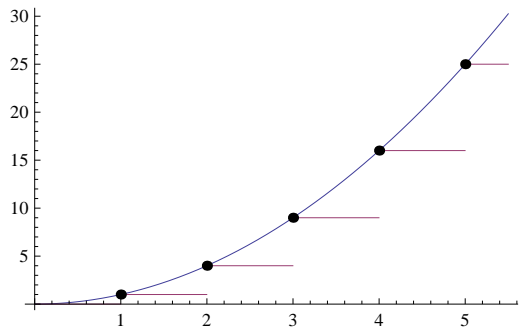}
\caption{\small The functions $V_{\ref{ex 7+1}}$ and $g_{\ref{ex 7+1}}$. The points of $M$ are denoted by black dots. }\label{fig 7+1}
\end{center}
\end{figure}
\end{example}

\section{Singular control problem}\label{sec control}

In this section we shall show that the principles behind Beibel-Lerche -method can be used in a singular control situation resulting in conclusions similar to the ones in the previous sections. Moreover, this is the first time, that the author is aware of, that the Beibel-Lerche -method is used in a singular control situation. 

\subsection{Definitions}
Let us now consider a controlled diffusion 
\begin{align*}
dX^Z_t=\mu(X_t^Z)dt + \sigma(X_t^Z)dW_T-dZ_t,\quad X^Z_0=X_0=x,
\end{align*}
on $\R_+$, where $\mu$ and $\sigma$ are as previously and $Z_t$ is non-negative, non-decreasing, right-continuous, and $\mathcal{F}_t$-adapted process. Consequently any admissible control has finite variation. An arbitrary control can be composed as $Z_t=Z^C_t+\sum_{0\leq s\leq t}\Delta Z_{s}$, where $Z^C_t$ is the continuous part and $\Delta Z_{s}=Z_{s+}- Z_{s-}$ is the jump part of the control at time $s$.

Let $g:\R_+\to\R$ be as previously and let us study a singular control problem
\begin{align}\label{eq control problem}
W_Z(x)=\sup_Z\E_x\left\{\int_0^\infty e^{-rs}g(X^Z_s) \circ dZ_s\right\}.
\end{align}
Here we understand (in the spirit of e.g. \cite{JaJoZe08})
\begin{align*}
\int_0^T e^{-rs}g(X^Z_s) \circ dZ_s=\int_0^T e^{-rs}g(X_s^Z)dZ^C_t+\sum_{0\leq t \leq T}\int_0^{\Delta Z_t}e^{-rt}g(X_t^Z-u)du.
\end{align*}

If $X^Z_t$ exits the state space at time $\zeta$, we understand $g\equiv 0$ for all $t\geq\zeta$. In this problem setting, we understand an optimal control to be any control $Z^*$ that produces the value $W_Z$. We denote by $C\subset\R_+$ and $S=\R_+\setminus C$ the inaction region and action region, respectively, of an optimal control. Furthermore, we denote by ${M'}:=\left\{x\mid x=\argmax\{g(z)/\psi'(z)\}\right\}$ the set of global maximum points of $g/\psi'$ and let ${z'}^*=\sup\{{M'}\}$ be the maximal element of that set. Moreover, we denote by $Z^b$ the control of reflecting downwards at the threshold $b$. It is known that, for $x<b$, the value applying $Z^b$ can be written as \(\frac{g(b)}{\psi'(b)}\psi(x)\) (see e.g. discussion below Lemma 3.1 in \cite{Matomaki12}).

To end this introductory subsection, let us present an auxiliary lemma, which is a consequence of It\^o's Lemma (cf. Lemma 3.1 from \cite{Matomaki12}). 

\begin{lemma}\label{lemma cont value}
Let $f:\R_+\to\R$ be a twice continuously differentiable function such that $\sigma(x)f'(x)$ is bounded on $(0,\varepsilon)$ for some $\varepsilon>0$. Let $Z_t$ be an arbitrary admissible control such that $X^Z_t$ is bounded. Then
\begin{align*}
f(x)=\E_x\left\{\int_0^\infty e^{-rs}\left(r-\mathcal{A}\right)f(X^Z_s)ds+\int_0^\infty e^{-rs}f'(X^Z_s)dZ^C_s+\sum_{0\leq s} e^{-rs}\Delta f(X^Z_s) \right\}.
\end{align*}
\end{lemma}

\begin{proof}
Let us apply (generalised) It\^o's lemma to a mapping $e^{-rt }f(X^Z_t)$ to get
\begin{align}\label{eq f}
\begin{aligned}
e^{-r\tau}f(X^Z_t)=f(x)&+\mathcal{E}_t+\int_0^te^{-rs}\left[(\mathcal{A}-r)f(X^Z_s)ds-f'(X^Z_s)dZ^C_s\right]\\
&-\sum_{0\leq s\leq t} e^{-rs}\Delta f(X^Z_s),
\end{aligned}
\end{align}
where $\mathcal{E}_t=\int_0^t e^{-rs}\sigma(X^Z_s)f'(X^Z_s)dW_s$ is a local martingale (e.g. Theorem IV.30.7 in \cite{RogWil00v2}). The boundedness of $X^Z$ implies that also $f(X^Z_t)$ is bounded so that $\lim_{t\to\infty}e^{-rt}f(X^Z_t)=0$. Moreover, the boundedness of $X^Z$ together with the fact that $\sigma(x)f'(x)$ is bounded near $0$, implies that $\E_x\{\mathcal{E}_t\}=0$. Therefore, the claim follows by taking expectation of both sides in \eqref{eq f} and letting $t\to\infty$.
\end{proof}

\subsection{The main results} A use of a control $Z^b$ results into a value $\frac{g(b)}{\psi'(b)}\psi(x)$, which clearly resembles the value of a one-sided optimal stopping problem, only now we have a ratio $\frac{g}{\psi'}$ instead of $\frac{g}{\psi}$. Hence it is no surprise that the analysis of a singular control problem follows more or less analogous path to that of an optimal stopping problem. 
  
\begin{lemma}\label{lemma cont}
Assume that ${z'^*}>0$ and let $a,b\in {M'}$, $b\neq0$, and $0< a\leq b\leq z'^*<\infty$ and let $x\leq b$. Then the following controls yield the same value $\frac{g(z'^*)}{\psi'(z'^*)}\psi(x)$:\label{lemma cont A}
	\begin{enumerate}[(i)]
	\item Control $Z^{z'^*}$.
	\item Control $Z^{b}$.
	\item If $x\in(a,b)$, wait until $X_t$ hits either $a$ or $b$ and then reflect downwards at a threshold it hits first. (I.e. wait time $\tau_{(a,b)}$ and then reflect downwards at a threshold $X_{\tau_{(a,b)}}$.)
	\item Wait until $X_t$ enters the set $M'^+$, where where $M'^+\subset M'$ is any subset that contains $z'^*$, and use downward control with a reflecting threshold $X_{\tau_{M'^+}}$. \label{lemma cont Aiv}
\end{enumerate}
\end{lemma}

\begin{proof}
\noindent \textbf{(i), (ii) } Since $b,z'^*\in M'$, we have, for $x<b$, 
\begin{align*}
\E_x\left\{\int_0^\infty e^{-rs} g(X^Z_s)\circ dZ^{z'^*}_s\right\}=\frac{g(z'^*)}{\psi'(z'^*)}\psi(x)=\frac{g(b)}{\psi'(b)}\psi(x)=\E_x\left\{\int_0^\infty e^{-rs} g(X^Z_s)\circ dZ^{b}_s\right\}.
\end{align*}

\noindent \textbf{(iii)} Let $x\in(a,b)$. The proposed control gives value
\begin{align*}
&\E_x\left\{e^{-r{\tau_a}  };\tau_a<\tau_b\right\}\E_a\left\{\int_0^\infty e^{-rs}g(X^Z_s)\circ dZ^a\right\}\\
&\qquad +\E_x\left\{e^{-r{\tau_b}  };\tau_b<\tau_a\right\}\E_b\left\{\int_0^\infty e^{-rs}g(X^Z_s)\circ dZ^b\right\}\\
&=\frac{\psi(b)\varphi(x)-\varphi(b)\psi(x)}{\psi(b)\varphi(a)-\varphi(b)\psi(a)}\left(\frac{g(a)}{\psi'(a)}\psi(a)\right)+ \frac{\psi(x)\varphi(a)-\varphi(x)\psi(a)}{\psi(b)\varphi(a)-\varphi(b)\psi(a)}\left(\frac{g(b)}{\psi'(b)}\psi(b)\right)\\
&=\frac{\psi(b)\psi(a)\frac{g(a)}{\psi'(a)}-\psi(a)\psi(b)\frac{g(b)}{\psi'(b)}}{\psi(b)\varphi(a)-\varphi(b)\psi(a)}\varphi(x)+ \frac{\psi(b)\varphi(a)\frac{g(b)}{\psi'(b)}-\varphi(b)\psi(a)\frac{g(a)}{\psi'(a)}}{\psi(b)\varphi(a)-\varphi(b)\psi(a)}\psi(x).
\end{align*}
As $a,b\in M$, we see straight that the coefficient of $\varphi(x)$ vanishes, and further that the last term can be written as
\begin{align*}
\left(\frac{\psi(b)\varphi(a)-\varphi(b)\psi(a)}{\psi(b)\varphi(a)-\varphi(b)\psi(a)}\right)\frac{g(b)}{\psi'(b)}\psi(x)=\frac{g(b)}{\psi'(b)}\psi(x).
\end{align*}

\noindent \noindent \textbf{(iv)} As a diffusion is continuous, we must have $\tau_{M'^+}=\tau_{\{a,b\}}$, where $b$ be the smallest point of $M'^+$ such that $b\geq x$, and $a$ is the greatest point of $M'^+$ such that $x\leq a$ (if no such $a$ exist, then $\tau_{M'^+}=\tau_b$). Consequently the claim follows from item (iii) (or (ii))
\end{proof}

Let us state our main result. This utilises  Lemma \ref{lemma cont value}, and hence we shall make the following assumption.
\begin{assumption}\label{assumption}
Assume that $\sigma(x)\psi'(x)$ is bounded on $(0,\varepsilon)$ for some $\varepsilon>0$ and let us consider only controls $Z$ such that $X^Z$ is bounded from above.
\end{assumption}

\begin{theorem}\label{prop cont}
\begin{enumerate}[(A)]
	\item Let Assumption \ref{assumption} hold. For $x\leq z'^*$, the value can be written as
	\begin{align}
	W_Z(x)=\frac{g(z'^*)}{\psi'(z'^*)}\psi(x).
	\end{align}
	Moreover, an admissible optimal control exists and it is any reflecting downwards -control that leads to this value (cf. Lemma \ref{lemma cont}). 	
	\label{prop cont a}
\item If $z'^*=0$, then there is no threshold $\varepsilon>0$ such that a downwards reflecting control $Z^\varepsilon$ would yield the maximal value for $x\leq \varepsilon$.\label{prop cont b} 
\end{enumerate}	 
\end{theorem}

\begin{proof}
\noindent \textbf{(A)} Let $x\leq z'^*$ and let $Z_t$ be an arbitrary admissible control such that $X^Z_t$ is bounded from above. Then we have
\begin{align*}
\E_x\left\{\int_0^\infty e^{-rs}g(X^Z_s)\circ dZ_s\right\}&= \E_x\left\{\int_0^\infty e^{-rs}g(X^Z_s)dZ^C_s+\sum_{0\leq s}e^{-rs}\int_{X^Z_{s+}}^{X^Z_{s-}}g(u)du\right\}\\
&=\E_x\left\{\int_0^\infty e^{-rs}\frac{g(X^Z_s)}{\psi'(X^Z_s)}\psi'(X^Z_s)dZ^C_s\right.\\
&\qquad\qquad\left.+\sum_{0\leq s}e^{-rs}\int_{X^Z_{s+}}^{X^Z_{s-}}\frac{g(u)}{\psi'(u)}\psi'(u)du\right\}\\
&\leq \frac{g(z'^*)}{\psi'(z'^*)}\E_x\left\{\int_0^\infty e^{-rs}\psi'(X^Z_s)\circ dZ_s\right\}
=\frac{g(z'^*)}{\psi'(z'^*)}\psi(x),
\end{align*}
where the last equality follows from Lemma \ref{lemma cont value}. For $x\leq z'^*$ this value is attained by applying a reflecting control $Z^{z^*}$, and so the proposed $W_Z(x)$ is the value function. Furthermore, any control from Lemma \ref{lemma cont} produces this maximal value.

\noindent \textbf{(B)}   Let $M'=\{0\}$ and suppose, contrary to our claim, that there exists $\varepsilon>0$ such that, for $x\in(0,\varepsilon)$, a singular control $Z^\varepsilon$ provides the maximal value, so that $W_Z(x)=\frac{g(\varepsilon)}{\psi'(\varepsilon)}\psi(x)$. Since $\limsup_{z\to0}\frac{g(z)}{\psi'(z)}$ is the maximum for $g/\psi'$, there exists $\hat{x}\in(0,\varepsilon)$ such that $\frac{g(\hat{x})}{\psi'(\hat{x})}>\frac{g(\varepsilon)}{\psi'(\varepsilon)}$, whence for all $x\in(0,\hat{x})$
\begin{align*}
\E_x\left\{\int_0^\infty e^{-rs  }g(X^Z_s) \circ dZ^{\hat{x}}\right\}&=\frac{g(\hat{x})}{\psi'(\hat{x})}\psi(x) >\frac{g(\varepsilon)}{\psi'(\varepsilon)}\psi(x)\\
&=\E_x\left\{\int_0^\infty e^{-rs  }g(X^Z_s)\circ dZ^\varepsilon\right\}=W_Z(x),
\end{align*}
contradicting the maximality of $W_Z(x)$.
\end{proof}

The most noteworthy fact is that the Beibel-Lerche -approach allows us to construct a value for a control problem locally, whereas usually, applying variational inequalities, one needs global information in order to solve a problem.

Furthermore, we see that the optimal control is not uniquely determined if $M'$ has at least two members. In fact we could use following kind of control: Choose $M'^+_n$ to be an arbitrary sequence of subsets of $M'$ that includes $z'^*$. In the first step wait time $\tau_{M_0'^+}$ and then reflect downwards at $X_{\tau_{M_0'^+}}$ until $X_t\notin M_0'^+$. In the second step wait time $\tau_{M_1'^+}$ and then reflect downwards at $X_{\tau_{M_1'^+}}$ until $X_t\notin M_1'^+$, etc. This control also leads to a value $\frac{g(z'^*)}{\psi'(z'^*)}\psi(x)$. 

However, in the sequel we wish to be unambiguous, and consistent with the optimal stopping scene, and hence we select $M'$ to be our optimal action region on interval $(0,z'^*]$ with a notion that there might also be others.

\subsection{Minor results}

Similar to optimal stopping scheme, the main theorem gives handful of corollaries in singular control case .

If $g(0+)<0$ or $\frac{d}{dx}\frac{g}{\psi'}(0+)>0$, then $z'^*>0$ and we can directly apply the following corollary.

\begin{corollary}\label{cor cont 0}
Let Assumption \ref{assumption} hold. If $z'^*>0$, then $(0,z'^*)\setminus M'\subset C$ and the value reads as $\frac{g(z'^*)}{\psi'(z'^*)}\psi(x)$ for $x\leq z'^*$.
\end{corollary}

\begin{proof}
Straight consequence of Theorem \ref{prop cont}.
\end{proof}

The special cases that $0$ or $\infty$ is contained in $M'$ are handled in the following two corollaries.

\begin{corollary}\label{cor cont 1}
Let $0\in M'$.
\begin{enumerate}[(A)]
	\item  One can never use control $Z^0$, i.e. one cannot reflect downwards at $0$.\label{cor cont 1 a}
	\item If $z'^*=0$, then one of the following is true.\label{cor cont 1 b}
	\begin{enumerate}[(i)]
	\item It is optimal to drive the process instantaneously (i.e. infinitely fast) to the boundary $0$, whence $S=\R_+$.
	\item There exists $\varepsilon>0$ such that $(0,\varepsilon]\subset S$ and $C\neq \emptyset$.
	\item The optimal control is something else than an admissible reflecting control satisfying Assumption \ref{assumption}.
\end{enumerate}
\end{enumerate}	
\end{corollary}

\begin{proof}

\textbf{(A)} If $0$ is unattainable, it is never attained at a finite time and thus $Z^0$ cannot be used as a control. If $0$ is attainable (i.e. exit or killing), then the process is terminated immediately at $0$, before $Z^0$ is activated, and thus $Z^0$ cannot be used. 

\textbf{(B)}  Straight consequence of Theorem \ref{prop cont}. 
\end{proof}

\begin{corollary}\label{cor cont 1b}
	 Let Assumption \ref{assumption} hold and assume that $z'^*=\infty$. Then $W_Z(x)=A\psi(x)$ for all $x\in\R_+$, where $A=\limsup_{x\to\infty}\left\{\frac{g(x)}{\psi'(x)}\right\}$. Especially, if $A=\infty$, then $W_Z(x)\equiv \infty$. Moreover, \label{cont B}
	\begin{enumerate}[(i)]
	\item if $M'=\{z'^*\}=\{\infty\}$, then $C=\R_+\setminus\{\infty\}$ and there is no admissible optimal control;
	\item if there is at least one other element in $M'$, then there exists an admissible optimal control for $x\leq b'^*:=\sup \{M'\setminus\{\infty\}\}$ and no admissible optimal control for $x>b'^*$, and $C=\R_+\setminus M'$.
	\end{enumerate}
\end{corollary}

\begin{proof}
Let $A=\limsup_{x\to\infty}\{\frac{g}{\psi'}\}$. Using the arguments from the proof of Theorem \ref{prop cont}, we known that for all admissible controls $Z$ under Assumption \ref{assumption} it is true that
\[\E_x\left\{\int_0^\infty e^{-rs  }g(X^Z_s)\circ dZ_s\right\}\leq A \psi(x).\]

 Since $z'^*=\infty$, we can choose an increasing sequence $z_n\in\R_+$, such that $\lim_{n\to\infty}z_n=z'^*$, $\frac{g(z_n)}{\psi'(z_n)}$ is increasing, and $\lim_{n\to\infty}\frac{g(z_n)}{\psi'(z_n)}=A$. Then a sequence of controls $Z^{z_n}$ gives an increasing sequence of values  $\frac{g(z_n)}{\psi'(z_n)}\psi(x)$. As this converges to $A\psi(x)$, it must be the maximal value. 
 
\textbf{(i)} If $\infty$ is unattainable, then it is never reached and there is no control that provides the optimal value $A\psi(x)$. If $\infty$ is attainable (i.e. exit or killing), then the process is terminated immediately at $\infty$ before a control $Z^\infty$ is activated. 

\textbf{(ii)} Straight consequence from Theorem \ref{prop cont} and part (i). 
\end{proof}
Notice that by Corollary \ref{cor cont 1b} the condition $z'^*=\infty$ alone is not enough to guarantee that the value cannot be attained with an admissible control, while condition $M'=\{\infty\}$ is enough. 

\begin{lemma}\label{cor cont 2}
If an interval $(a,b)\subset M$, then $g(x)=K \psi'(x)$ for all $x\in(a,b)$ for some $K>0$. 
\end{lemma}

\begin{proof}
Since $(a,b)\subset M$ we know that $x=\argmax\{g(z)/\psi'(z)\}$ for all $x\in(a,b)$. This in turns means that $\frac{g(x)}{\psi'(x)}=K$ for some $K>0$ on $(a,b)$.
\end{proof}

\subsection{Differences to optimal stopping}

The main difference to the optimal stopping case is the fact that in the singular control problem \eqref{eq control problem} only the ratio $g/\psi'$ has meaning; the ratio $g/\varphi'$ leads to values that cannot be attained with downward control. It should be mentioned that this cannot be fixed easily by adding an upward control to the problem. This is because when both downward and upward controls are present, the solution can rarely be identified with a one-sided control, and as a consequence the analysis of the simple ratios $g/\psi'$ and $g/\varphi'$ is not adequate. 

We also notice that we found more corollaries in the optimal stopping scene. This can be seen due to a rather simple local characterisation of stopping regions/continuation regions in optimal stopping problems: $x$ is in the continuation region whenever $V(x)>g(x)$. On the other hand, in control problems action regions/inaction regions are rarely found locally; More often than not, they are found globally applying variational inequalities. Most noticeably this is seen in the fact that in the control scene the set $C_{\psi'}=\left\{x\mid \frac{g(x)}{\psi'(x)}\text{ is strictly increasing }\right\}$ is not necessarily part of the inaction region, as we shall see in Example \ref{ex cont increasing}.

\subsection{Extensions}\label{subsec con gen}

We can easily extend the results from this section to concern also a running payoff case. To that end let $\pi:\R_+\to\R$ be once continuously differentiable function for which $\E_x\left\{\int_0^\infty|\pi(X_s)|ds\right\}<\infty$, and let us study a problem
\begin{align}\label{eq cont gen}
\sup_Z\E_x\left\{\int_0^\infty e^{-rs} \pi(X^Z_s)ds+\int_0^\infty e^{-rs}g(X^Z_s)\circ dZ_s\right\}.
\end{align}
As the resolvent $(R_r\pi)(x)$ solves the ordinary differential equation $\left(\mathcal{A}-r\right)u(x)=-\pi(x)$, a straight consequence of Lemma \ref{lemma cont value} is that 
\begin{align*}
(R_r\pi)(x)=\E_x\left\{\int_0^\infty e^{-rs} \pi(X^Z_s)ds+\int_0^\infty e^{-rs}(R_r\pi)'(X^Z_s)\circ dZ_s\right\}.
\end{align*}
Hence the problem \eqref{eq cont gen} can be re-written as
\begin{align*}
(R_r\pi)(x)+\sup_Z\E_x\left\{\int_0^\infty e^{-rs}\left(g(X^Z_s)-(R_r\pi)'(X^Z_s)\right)\circ dZ_s\right\},
\end{align*}
provided that the conditions of Lemma \ref{lemma cont value} are satisfied.

It follows at once that all the results from this section hold for this problem with obvious changes and with the set $M'=\left\{x\mid x=\argmax\{\frac{g(z)-(R_r\pi)'(z)}{\psi'(z)}\}\right\}$.

\subsection{Controlling upwards}
Until now we have only controlled the diffusion downwards. Let us now introduce an upward control defining
\begin{align*}
dX^Y_t=\mu(X_t^Y)dt+\sigma(X^Y_t)dW_t+dY_t,\quad X^Y_0=X_0=x,
\end{align*}
on $\R_+$, where $\mu$ and $\sigma$ are as previously and $Y_t$ is a non-negative, non-decreasing, right-continuous, and $\mathcal{F}_t$-adapted. For a function $g$, defined as previously, we define a singular control problem
\begin{align}\label{eq control problem 2}
W_Y(x)=\sup_Y\E_x\left\{\int_0^\infty e^{-rs  }g(X^Y_s)\circ dY_s\right\}.
\end{align}

It is known that, for $x>a$, the value applying $Y^a_t$, i.e. a control that reflects upwards at a threshold $a$, can be written as \(-\frac{g(a)}{\varphi'(a)}\varphi(x)\) (see e.g. discussion below Lemma 3.1 in \cite{Matomaki12}).
It is now quite clear that all the results from this section hold true for the problem \eqref{eq control problem 2} near the upper boundary with obvious changes; Instead of $g/\psi'$ we have $-g/\varphi'$, instead of $M'$ we have $N'=\left\{x\mid x=\argmax\{-g(y)/\varphi'(y)\}\right\}$, instead of $z'^*$ we have $y'^*=\inf\{N'\}$, etc.

\subsection{Examples}\label{sec example singular}

Let $X_t$ be as in Section \ref{sec example}, i.e. $X_t$ is a geometric Brownian motion for which $\psi(x)=x^2$ and $\varphi(x)=x^{-6}$.

\begin{example}\label{ex 8} Here we show that both an impulse control and a singular control can yield the maximal value. Let 
\begin{align*}
g_{\ref{ex 8}}(x)=\begin{cases}
16\left(\sqrt{x}-2\right),\quad &x\leq16\\
2x,\quad &x\in(16,25)\\
20\left(\sqrt{x}-2.5\right),\quad &x\geq25,
\end{cases}
\end{align*}
so that $g_{\ref{ex 8}}$ is continuous, increasing, $\sup\{\frac{g_{\ref{ex 8}}(x)}{\psi'(x)}\}=1$, and $M'=[16,25]$. Consequently, by Theorem \ref{prop cont}, for $x\leq 25$, a control $Z^{25}$ (with $M'^+=\{25\}$) provides the solution and the value reads as $W_Z(x)=\psi(x)$. 

On the other hand, taking $M'=[16,25]$ to be the action region, we get a control that coincides with an impulse control $(\tau_{25};\zeta)$ with $\tau_{25}=\inf\{t\geq0\mid X_t=25\}$ and $\zeta=9$. That is, every time we hit the state $25$, we jump to the state $16$. It can be easily calculated that, for $x\leq 25$, this impulse control also results into a value $\psi(x)$.
\end{example}

\begin{example}\label{ex cont increasing} In this example we will see that $C_{\psi'}=\{x\mid \frac{g(x)}{\psi'(x)} \text{ is strictly increasing }\}$ does not necessarily belong to the inaction region. Take
\begin{align*}
g_{\ref{ex cont increasing}}(x)=\begin{cases}
\sqrt{x}-1,\quad &x<6.25\\
1.4\sqrt{x}-2,\quad&x\geq 6.25.
\end{cases}
\end{align*}
Now $M'=\{z^*\}=\{4\}$, and $C_{\psi'}=(0,4)\cup(6.25,8.16)$. However, it can be shown (cf. Lemma 1 in \cite{Alvarez00b} for a verification result) that the action region is $[4,\infty)$ and that the optimal control is $Z^4$.  In other words, contrary to what Lemma \ref{lemma increasing} from optimal stopping scene would suggest, the region $C_{\psi'}$ do not necessary belong to the inaction region in singular control problems.
\end{example}

\section{Connection between singular control and optimal stopping}\label{sec connection}

\subsection{Introducing the associated optimal stopping problem}

In most studies analysing the connection between singular control and optimal stopping there are some growth, convexity, or positiveness restrictions on the payoff $g$, or there are some restrictions for the diffusion process. These restrictions guarantee the connection to hold everywhere on the state space which is exactly what one usually needs or wants. However, in some sense something has gone unnoticed because of these restrictions: the renowned connection is not a global phenomenon. Here we will show that when we let the setting to be very general, the connection does not necessarily hold on the whole state space (Proposition \ref{lemma W'} and Example \ref{example last}). 

To introduce the connection in the present case, let us consider control problems \eqref{eq control problem} and \eqref{eq control problem 2} and let $X$, $X^Z$, $X^Y$, $Z$, $Y$, $W^Z$, $W^Y$ and $g$ be as in the previous section.  

Let us define an associated diffusion
\[d\hat{X}_t=\left(\mu(\hat{X}_t)+\sigma'(\hat{X}_t)\sigma(\hat{X}_t)\right)dt+ \sigma(\hat{X}_t)dW_t\]
and let us consider an optimal stopping problem
\begin{align}\label{eq as stop}
\hat{V}(x)=\sup_\tau\E_x\left\{e^{-\int_0^\tau\left(r-\mu'(\hat{X}_s)\right)ds}g(\hat{X}_\tau)\right\}.
\end{align}
(The infinitesimal generator $\hat{\mathcal{A}}-\left(r-\mu'(x)\right)$ can be seen as a derivative of the operator $\mathcal{A}-r$.)

In the sequel we need to utilise the fact that the Laplace transform of $\tau_b=\inf\{t\geq0\mid X_t=b\}$, the hitting time to a state $b$, with respect to a diffusion $\hat{X}$ can be written as 
\begin{align}\label{eq laplace}
\E_x\left\{e^{-\int_0^{\tau_b}\left(r-\mu'(\hat{X}_s)\right)ds}\right\}=\begin{cases}\frac{\psi'(x)}{\psi'(b)},\quad&x\leq b\\
\frac{\varphi'(x)}{\varphi'(b)},\quad&x> b.
\end{cases}
\end{align}
It is quite clear that if $\psi$ and $\varphi$ are convex, then their derivatives $\psi'$ and $-\varphi'$ are the non-negative increasing and decreasing fundamental solutions to $\left(\hat{\mathcal{A}}-\left(r-\mu'(x)\right)\right)u(x)=0$. This in turn ensures that in the convex case condition \eqref{eq laplace} holds. 

In the following we state sufficient conditions guaranteeing that the fundamental solutions $\psi$ and $\varphi$ are convex.

\begin{lemma}
\begin{enumerate}[(A)]
\item Assume that a transversality condition $\lim_{t\to\infty}\E_x\left\{e^{-rt}X_t\right\}=0$ holds and that $r>\mu'(x)$ for all $x\in\R_+$. Further assume that either $0$ is unattainable or $\lim_{x\to0}\left(rx-\mu(x)\right)\geq0$. Then $\psi(x)$ is convex.\label{lemma con A}
	\item Assume that one or the other of the following hold.
	\begin{enumerate}[(i)]
	\item $\mu(x)\geq0$ for all $x\in\R_+$; or
	\item a transversality condition holds, $\infty$ is unattainable and $r>\mu'(x)$ for all $x\in\R_+$.
	\end{enumerate}  Then $\varphi$ is convex.\label{lemma con B}
\end{enumerate}
\end{lemma}
\begin{proof}
Item \eqref{lemma con A} and the latter part of item \eqref{lemma con B} follow from similar deduction to Corollary 1 from \cite{Alvarez03}. The former part of item \eqref{lemma con B} follows after noticing that
\[\frac{1}{2}\sigma^2(x)\varphi''(x)=r\varphi(x)-\mu(x)\varphi'(x).\qedhere\]
\end{proof}

\subsection{The connection between singular control and optimal stopping}\label{subsec covex}

Assuming that the condition \eqref{eq laplace} holds, we can straightforwardly apply Theorem \ref{prop 1} to the associated stopping problem \eqref{eq as stop} with the sets $\hat{M}=\left\{x\mid x=\argmax\{\frac{g(z)}{\psi'(z)}\}\right\}$ and $\hat{N}=\left\{x\mid x=\argmax\{\frac{g(y)}{-\varphi'(y)}\}\right\}$. On the other hand, at the same time we can apply Theorem \ref{prop cont} to the control problems \eqref{eq control problem} and \eqref{eq control problem 2} with the sets $M'=\hat{M}$ and $N'=\hat{N}$.

These facts give the following proposition, which is the celebrated connection between a singular control problem and the associated optimal stopping problem in our case.

\begin{proposition}\label{lemma W'} Assume that $\psi$ and $\varphi$ satisfy \eqref{eq laplace} and that Assumption \ref{assumption} hold.
\begin{enumerate}[(A)]
	\item One has $W_Z'(x)=\hat{V}(x)$ for all $x\leq z'^*$.
	\item One has $W_Y'(x)=\hat{V}(x)$ for all $x\geq y'^*$.
\end{enumerate}
\end{proposition}

We see that the associated stopping problem carries, potentially, more information than a single singular control problem; Although $g/\varphi'$ played no role in a downward controlled singular control problem, it has a well defined meaning in the associated stopping problem. This is illustrated in Example \ref{example last} below. Moreover, the proposition reveals that there is no guarantee, \emph{a priori}, that we can connect a one-sided singular control problem to its associated stopping problem on \emph{the whole state space}. This observation gives us the following, quite interesting, necessary condition under which this connection can hold everywhere.

\begin{corollary}\label{cor WV}  Assume that $\psi$ and $\varphi$ satisfy \eqref{eq laplace} and that Assumption \ref{assumption} hold.
\begin{enumerate}[(A)]
	\item We can have $W_Z'(x)=\hat{V}(x)$ for all $x\in\R_+$ only if $y'^*=\infty$.
	\item We can have $W_Y'(x)=\hat{V}(x)$ for all $x\in\R_+$ only if $z'^*=0$.
\end{enumerate}
\end{corollary}

\begin{proof}
\noindent \textbf{(A)} By Proposition \ref{lemma W'} we have $\hat{V}(x)=\frac{g(y'^*)}{\varphi'(y'^*)}\varphi'(x)$, for all $x\geq y'^*$. As the value $\frac{g(y'^*)}{\varphi'(y'^*)}\varphi(x)$ cannot be reached by applying a downward control, we can have $W_Z'(x)=\hat{V}(x)$ for all $x\in\R_+$ only if $y'^*=\infty$. Part (B) follows analogously.
\end{proof}

\subsection{When fundamental solutions can be concave}
Earlier we have seen that the Laplace transform of the hitting time in \eqref{eq laplace} holds if $\psi$ and $\varphi$ were convex. Here we generalise this to a case where $\psi$ and $\varphi$ are concave near the boundaries. 

The following theorem (Theorem 1 in \cite{Alvarez03}) reveals that the fundamental solutions are concave near the boundaries, if $r-\mu'<0$ there.

\begin{theorem}\label{theo integral}
Assume that the transversality condition holds and that $0$ and $\infty$ are unattainable. Then for all $x\in\R_+$
\begin{align*}
\sigma^2(x)\frac{\psi''(x)}{S'(x)}&=2r\int_0^x\psi(y)\left(\theta(x)-\theta(y)\right)m'(y)dy,\quad\text{and }\\
\sigma^2(x)\frac{\varphi''(x)}{S'(x)}&=2r\int_x^\infty\varphi(y)\left(\theta(y)-\theta(x)\right)m'(y)dy,
\end{align*}
where $\theta(x)=rx-\mu(x)$.
\end{theorem}
The following sums up the Laplace transform of the hitting time in this concave case.

\begin{lemma}\label{lemma lap}
\begin{enumerate}[(A)]
	\item Assume that the transversality condition holds and that $0$ is unattainable for a diffusion $X_t$. Further, assume that there exists $\varepsilon>0$ such that $r-\mu'(x)<0$ for all $x\in(0,\varepsilon)$ (it may be negative also elsewhere). Let $b>0$ and $x\in(0,b)$. Then
	\begin{align*}
	\E_x\left\{e^{-\int_0^{\tau_b}\left(r-\mu'(\hat{X}_s)\right)ds}\right\}=\frac{\psi'(x)}{\psi'(b)},
	\end{align*}
	where $\tau_b=\inf\{t\geq0\mid \hat{X}_t= b\}$ is the first hitting time to a state $b$.\label{lemma lap a}
	\item Assume that the transversality condition holds and that $\infty$ is unattainable for a diffusion $X_t$. Further, assume that there exists $H<\infty$ such that $r-\mu'(x)<0$ for all $x\in(H,\infty)$ (it may be negative also elsewhere). Let $a<\infty$ and $x\in(a,\infty)$. Then
	\begin{align*}
	\E_x\left\{e^{-\int_0^{\tau_a}\left(r-\mu'(\hat{X}_s)\right)ds}\right\}=\frac{\varphi'(x)}{\varphi'(b)},
	\end{align*}
	where $\tau_a=\inf\{t\geq0\mid \hat{X}_t= a\}$ is the first hitting time to a state $a$.\label{lemma lap b}
\end{enumerate}
\end{lemma}

\begin{proof} The proof follows quite closely that of Theorem 9 in \cite{Alvarez01b}.

\textbf{(A)} Now $\varphi''(x)\psi'(x)-\psi''(x)\varphi'(x)=2rB\hat{S}'(x)>0$, where $\hat{S}'(x)=S'(x)/\sigma^2(x)$ is the scale derivative of the process $\hat{X}$. Because of this, we see that the ratio $\varphi'(x)/\psi'(x)$ is increasing. Moreover, as $\psi'$ and $\varphi'$ are two independent solutions to $\left(\hat{\mathcal{A}}-(r-\mu'(x))\right)u(x)=0$, any solution to it can be expressed as $c_1\psi'(x)+c_2\varphi'(x)$ for some $c_1,c_2\in\R$. It is now an easy exercise in linear algebra to demonstrate that if $x\in(a,b)$, then 
\begin{align}\label{eq yrite}
G(x;a,b):=\E_x\left\{e^{-\int_0^{\tau_{(a,b)}}\left(r-\mu'(\hat{X}_s)\right)ds}\right\}= \frac{\varphi'(x)-\frac{\varphi'(b)}{\psi'(b)}\psi'(x)}{\varphi'(a)-\frac{\varphi'(b)}{\psi'(b)}\psi'(a)}+ \frac{\psi'(x)-\frac{\psi'(a)}{\varphi'(a)}\varphi'(x)}{\psi'(b)-\frac{\psi'(a)}{\varphi'(a)}\varphi'(b)},
\end{align}
where $\tau_{(a,b)}=\inf\left\{t\geq0\mid \hat{X}_t\notin (a,b)\right\}$ is the first exit time of $\hat{X}$ from an open interval $(a,b)$. Invoking the alleged boundary conditions of $X$ implies that 
\begin{align*}
\frac{\psi'(x)-\frac{\psi'(a)}{\varphi'(a)}\varphi'(x)}{\psi'(b)-\frac{\psi'(a)}{\varphi'(a)}\varphi'(b)} =\frac{\psi'(x)-\frac{\psi'(a)/S'(a)}{\varphi'(a)/S'(a)}\varphi'(x)}{\psi'(b)-\frac{\psi'(a)/S'(a)}{\varphi'(a)/S'(a)}\varphi'(b)}\to \frac{\psi'(x)}{\psi'(b)}\quad \text{as $a\to0$.}
\end{align*}
Consider now the first term on the right hand side of \eqref{eq yrite}. We want to show that it convergences to $0$ as $a$ approaches $0$. To that end, we firstly notice that it can be written as
\begin{align*}
\frac{\psi'(x)}{\psi'(a)}\frac{\frac{\varphi'(x)}{\psi'(x)}-\frac{\varphi'(b)}{\psi'(b)}}{\frac{\varphi'(a)}{\psi'(a)}-\frac{\varphi'(b)}{\psi'(b)}}\geq0,
\end{align*}
where the inequality follows from the fact that $\varphi'/\psi'$ is increasing and $\psi'>0$. Secondly, from Theorem \ref{theo integral} we see at once that $\psi$ is concave on $(0,\varepsilon)$, whence we can approximate
\begin{align*}
\frac{\psi'(x)}{\psi'(a)}\frac{\frac{\varphi'(x)}{\psi'(x)}-\frac{\varphi'(b)}{\psi'(b)}}{\frac{\varphi'(a)}{\psi'(a)}-\frac{\varphi'(b)}{\psi'(b)}}
=\frac{\psi'(x)}{\psi'(\varepsilon)}\frac{\psi'(\varepsilon)}{\psi'(a)}\frac{\frac{\varphi'(x)}{\psi'(x)}-\frac{\varphi'(b)}{\psi'(b)}}{\frac{\varphi'(a)}{\psi'(a)}-\frac{\varphi'(b)}{\psi'(b)}}\leq
\frac{\psi'(x)}{\psi'(\varepsilon)}\frac{\frac{\varphi'(x)}{\psi'(x)}-\frac{\varphi'(b)}{\psi'(b)}}{\frac{\varphi'(a)}{\psi'(a)}-\frac{\varphi'(b)}{\psi'(b)}},\quad\text{for all }a<\varepsilon.
\end{align*}
Letting $a\to0$ and invoking again the alleged boundary conditions of $X$ we get 
\begin{align*}
\frac{\psi'(x)}{\psi'(\varepsilon)}\frac{\frac{\varphi'(x)}{\psi'(x)}-\frac{\varphi'(b)}{\psi'(b)}}{\frac{\varphi'(a)}{\psi'(a)}-\frac{\varphi'(b)}{\psi'(b)}}\to \frac{\psi'(x)}{\psi'(\varepsilon)}\frac{\frac{\varphi'(x)}{\psi'(x)}-\frac{\varphi'(b)}{\psi'(b)}}{-\frac{\varphi'(b)}{\psi'(b)}}\leq0,
\end{align*}
indicating that the first term on the right hand side of \eqref{eq yrite} tends to zero as $a$ tends to zero. Consequently $$\lim_{a\to0}G(x;a,b)=\E_x\left\{e^{-\int_0^{\tau_b}\left(r-\mu'(\hat{X}_s)\right)ds}\right\}=\frac{\psi'(x)}{\psi'(b)}.$$

\noindent \textbf{(B)} Proof is analogous to the part\eqref{lemma lap a}.
\end{proof}

\subsection{Examples}\label{subsec connection example}
Again, let $X_t$ be a geometric Brownian motion for which $\psi(x)=x^2$ and $\varphi(x)=x^{-6}$. Let us present perhaps the most striking example of the paper.

\begin{example}\label{example last}
In this example we will illustrate that the associated stopping problem can be associated problem to a two different singular control problem \emph{at the same time}. Let
\begin{align*}
g_{\ref{example last}}(x)=\begin{cases}
x^2,\quad&x\leq1\\
x^{-8},\quad&x> 1.
\end{cases}
\end{align*}
It is quite straightforward to show that
\[\hat{V}_{\ref{example last}}(x)=\begin{cases}
\frac{1}{2}\psi'(x)\quad &x\leq 1\\
-\frac{1}{6}\varphi'(x)\quad &x>1.
\end{cases}\]
Moreover, now $M'=\{1\}$ and $\sup\{g_{\ref{example last}}/\psi'\}=\frac{1}{2}$. Thus we can say that for $x\leq 1$ we have $W_Z(x)=\frac{1}{2}\psi(x)$ and $\hat{V}_{\ref{example last}}(x)=W'_Z(x)$. (Actually, applying Theorem 1 from \cite{Alvarez00b}, we can say that $Z^1$ is the optimal control on $\R_+$ and that $W_Z(x)=\int_1^xg_{11}(z)dz +\frac{1}{2}\psi(1)$ for $x\geq 1$.)

On the other hand, now we also have $N'=\{1\}$ with $\sup\{g_{\ref{example last}}/\varphi'\}=\frac{1}{6}$, so that $W_Y(x)=\frac{1}{6}\varphi(x)$ and $\hat{V}_{\ref{example last}}(x)=W'_Y(x)$ for $x\geq1$. (Again we can prove that, for $x\leq 1$, $W_Y(x)=\int_x^1g_{\ref{example last}}(y)dy +\frac{1}{6}\varphi(1)$.)

It follows that the value $\hat{V}_{\ref{example last}}$ of the associated optimal stopping problem is, in fact, the associated value function for \emph{two} different singular control problems on separate regions $(0,1)$ and $(1,\infty)$:
\begin{align*}
\hat{V}_{\ref{example last}}(x)=\begin{cases}
W'_Z(x)\quad &x<1\\
W'_Y(x)\quad & x>1.
\end{cases}
\end{align*}
 This means that the connection between the one-sided singular control and optimal stopping problem is in general a local property rather than global.
\end{example}

\begin{acknowledgements}
The author is grateful to Professor L.H.R. Alvarez for helpful conversations and valuable comments during the research. The author would also like to thank Professor M. Zervos for pointing out Example \ref{ex 6}. Lastly, S. Christensen is acknowledged for his helpful comments on the contents of the earlier draft of the paper.
\end{acknowledgements}

\bibliographystyle{amsplain}
\bibliography{NearBoundariesBib}

\end{document}